\documentclass[a4paper, reqno]{amsart}

\usepackage[utf8]{inputenc}
\usepackage[T1]{fontenc}
\usepackage{unicode-math}
\usepackage{amsmath}
\usepackage{bbm}
\usepackage{amsthm}
\usepackage{amssymb}
\usepackage{color}
\usepackage[margin=1.5in]{geometry}

\theoremstyle{plain}
\newtheorem{thm}{Theorem}[section]
\newtheorem{lem}[thm]{Lemma}
\newtheorem{prop}[thm]{Proposition}

\newtheorem{conj}[thm]{Conjecture}

\theoremstyle{definition}
\newtheorem{rmk}[thm]{Remark}

\numberwithin{equation}{section}

\newcommand{\overbar}[1]{\mkern 1.5mu\overline{\mkern-1.5mu#1\mkern-1.5mu}\mkern 1.5mu}

\newcommand{\bbslash}{\mathbin{\backslash\mkern-6mu\backslash}}

\newcommand{\tif}{\ \text{if}\ }
\newcommand{\tand}{\ \text{and}\ }

\newcommand{\totherwise}{\ \text{otherwise}\ }

\newcommand{\Res}{\mathrm{Res}}

\newcommand{\dOrb}{\operatorname{\partial O}}
\newcommand{\diag}{\operatorname{diag}}

\newcommand{\Hom}{\operatorname{Hom}}
\newcommand{\End}{\operatorname{End}}

\newcommand{\Lie}{\operatorname{Lie}}

\newcommand{\Spf}{\operatorname{Spf}}

\newcommand{\GL}{\operatorname{GL}}

\newcommand{\len}{\operatorname{len}}

\newcommand{\tensor}{\otimes}

\newcommand{\iso}{\cong}

\newcommand{\mbC}{\mathbb{C}}

\newcommand{\mbF}{\mathbb{F}}

\newcommand{\mbL}{\mathbb{L}}

\newcommand{\mbN}{\mathbb{N}}

\newcommand{\mbQ}{\mathbb{Q}}
\newcommand{\mbR}{\mathbb{R}}

\newcommand{\mbX}{\mathbb{X}}
\newcommand{\mbY}{\mathbb{Y}}
\newcommand{\mbZ}{\mathbb{Z}}

\newcommand{\mcN}{\mathcal{N}}
\newcommand{\mcO}{\mathcal{O}}

\newcommand{\mcX}{\mathcal{X}}
\newcommand{\mcY}{\mathcal{Y}}
\newcommand{\mcZ}{\mathcal{Z}}

\newcommand{\mfc}{\mathfrak{c}}

\newcommand{\mfg}{\mathfrak{g}}

\newcommand{\mfs}{\mathfrak{s}}

\newcommand{\mfu}{\mathfrak{u}}

\begin{document}

\title{On the Arithmetic Fundamental Lemma through Lie algebras}
\author{Andreas Mihatsch}

\maketitle

\section{Introduction}
In \cite{zhang1}, Wei Zhang introduces his so-called Arithmetic Fundamental Lemma conjecture (AFL) in the group formulation and verifies it in the case $n=3$. He also mentions a similar conjecture, but for Lie algebras. Since then he found a trick using the Cayley transform to reduce the AFL for groups to the AFL for Lie algebras in the case of non-degenerate intersection.

In the present work, we develop these ideas to give a simplified proof of the AFL for $n=3$ (and $q\geq 5$). The reduction to the Lie algebra is proved in Section 2. The verification of the AFL for $n=3$ is done in Section 3. In the end, our computation is simpler than the computation in \cite{zhang1} because it is easier to work with coordinates on the Lie algebra than with coordinates on the group.

We will now introduce the AFL conjecture in the coordinates we will use later. For a more systematic introduction, see \cite{RSZ} or \cite{zhang1}.

\subsection{Statement of the AFL}
Let $n\geq 1$ and let $F_0$ be a $p$-adic field with $p\neq 2$. Let $F/F_0$ be an unramified quadratic extension with rings of integers $\mcO_{F_0}\subset \mcO_F$. We fix a uniformizer $π\in \mcO_{F_0}$ and an element $τ\in \mcO_F^\times$ with $\mathrm{tr}(τ)=0$. Let $\breve F$ be the completion of the maximal unramified extension of $F$ and denote its residue field by $\mbF$. We denote by $v$ the normalized valuation of $F$ and define the quadratic character $η$ on $F^\times$ by $η(a)=(-1)^{v(a)}$. It extends the quadratic character $η_{F/F_0}$ on $F_0^\times$ from local class field theory. The Galois conjugation of $F/F_0$ is denoted by $a\mapsto \overbar{a}$ or by $σ$.

We now define the orbital integrals which appear in the statement of the AFL conjecture. Let $U$ be the unitary group for the hermitian structure on $F^n$ defined by the matrix
$$J:=\left(\begin{smallmatrix} -π & & & \\ & 1 & & \\ & & \ddots & \\ & & & 1\end{smallmatrix}\right),$$
and denote its Lie algebra by $\mfu$. In particular,
\begin{equation}\label{liealg}
\mfu(F_0) = \left.\left\{
\left(\begin{matrix} a & v \\ π{}^t\overbar{v} & A \end{matrix}\right)\in M_n(F)\ 
\right|\ a = -\overbar{a}, A = -{}^t\overbar{A}\right\}.
\end{equation}

We also define the symmetric space
$$S(F_0):=\{γ\in GL_n(F)\mid γ\overbar{γ} = 1\},$$
with tangent space at the identity
$$\mfs(F_0):=\{y\in M_n(F)\mid y + \overbar{y} = 0\}=τ\cdot \mathfrak{gl}_n(F_0).$$

The group $GL_{n-1}(F)$ acts on $M_n(F)$ by conjugation via the embedding $GL_{n-1}→GL_n,\ h\mapsto \diag(h,1)$. An element $x\in M_n(F)$ is said to be \emph{regular semi-simple} if its stabilizer for this action is trivial and if its orbit is Zariski closed. We denote by $U(F_0)_{\mathrm{rs}},\ \mfu(F_0)_{\mathrm{rs}},\ S(F_0)_{\mathrm{rs}}$ and $\mfs(F_0)_{\mathrm{rs}}$ the regular semi-simple elements which also lie in the respective subset of $M_n(F)$. Note that $S(F_0)$ and $\mfs(F_0)$ are stable under the action of $GL_{n-1}(F_0)$.

Two elements $γ\in S(F_0)$ and $g\in U(F_0)$ are said to \emph{match} if they are conjugate under $GL_{n-1}(F)$. The same definition applies to a pair $y\in \mfs(F_0)$ and $x\in \mfu(F_0)$. We refer the reader to \cite{zhang1} and \cite{zhang2} for more details about regular semi-simplicity, matching and the quotient $GL_{n-1}(F)\bbslash M_n(F)$.

For a regular semi-simple $γ\in S(F_0)_{\mathrm{rs}}$, for a function $f\in C^\infty_c(S(F_0))$ and for a complex parameter $s\in \mbC$, we define the orbital integral
$$O_γ(f,s):=\int_{GL_{n-1}(F_0)} f(h^{-1}γh)η(\det h)|\det h|^s dh$$
with derivative
$$\partial O_γ(f):=\left.\frac{d}{ds}\right|_{s=0} O_γ(f,s).$$
The same formulas apply for $y\in \mfs(F_0)_{\mathrm{rs}}$ and $f'\in C^\infty_c(\mfs(F_0))$ yielding $\partial O_y(f')$. In both cases the Haar measure is normalized such that $\mathrm{Vol}(GL_{n-1}(\mcO_{F_0})) = 1$.

The orbital integrals $O_γ(f)$ (resp.\! $O_y(f')$) transform with $η\circ \det$ under conjugation by $GL_{n-1}(F_0)$ on $γ$ (resp.\! $y$). This motivates the definition of the \emph{transfer factor}. Let $e:={}^t(0,\ldots,0,1)$ and define the transfer factor for $γ\in S(F_0)_{\mathrm{rs}}$ as
$$Ω(γ):=η\big(\det((γ^ie)_{i=0,\ldots,n-1})\big).$$
For $y\in \mfs(F_0)_{\mathrm{rs}}$, we define the transfer factor
$$ω(y):=η\big(\det((y^ie)_{i=0,\ldots,n-1})\big).$$
Note that $(γ^ie)_{i=0,\ldots,n-1}$ (resp.\! $(y^ie)_{i=0,\ldots,n-1}$) is always invertible for a regular semi-simple $γ$ (resp.\! $y$), see \cite[§2.1]{zhang1}. The Product $Ω(γ)O_γ(f)$ (resp.\! $ω(y)O_y(f)$) is now invariant under conjugation by $GL_{n-1}(F_0)$. This concludes the definition of the left hand sides of Conjectures \ref{conjaflgrp} and \ref{conjafllie} below.

To define the geometric side of the AFL, we now introduce the moduli space of \emph{unitary ($p$-divisible formal) $\mcO_F$-modules of signature $(1,n-1)$}, denoted by $\mcN_n$. Let $S$ be a scheme over $\Spf \mcO_{\breve F}$ (i.e.\! $S$ is a scheme over $\mcO_{\breve F}$ and $π$ is locally nilpotent on $S$). A \emph{unitary $\mcO_F$-module} over $S$ is a triple $(X,i,λ)$ where $X/S$ is a $p$-divisible formal $\mcO_{F_0}$-module, $i:\mcO_F→\End(X)$ is an action of $\mcO_F$ and $λ:X\rightarrow X^\vee$ is a principal polarization such that
\begin{equation}
\label{signature}i(a)^\vee\circ λ = λ\circ i(\overbar{a})\ \ \forall a\in \mcO_F.
\end{equation}
The unitary module $(X,i,λ)$ is said to have signature $(r,s)$ if
\begin{equation}
\mathrm{charpol}(i(a)\mid \Lie X)(T) = (T-a)^r(T-\overbar{a})^s\ \ \text{inside}\ \mcO_S[T].
\end{equation}

There is a unique (up to isomorphism) such unitary module of signature $(1,0)$ over $\mbF$, which we denote by $\mbY=(\mbY,i_{\mbY},λ_{\mbY})$. We let $\overbar{\mbY}:=(\mbY,i_{\mbY}\circ σ,λ_{\mbY})$ which has signature $(0,1)$. Then we define the \emph{framing object} $\mbX_n:=\mbY\times \overbar{\mbY}^{n-1}$ which has signature $(1,n-1)$. Here, the $\mcO_F$-action and the polarization are defined diagonally on the product.

For a scheme $S/\Spf \mcO_{\breve F}$, we let $\overbar{S}:=S\tensor_{\mcO_{\breve F}} \mbF$. Then we define $\mcN_n(S)$ as the set of isomorphism classes of tuples $(X,i,λ,ρ)$ where $(X,i,λ)/S$ is a unitary $\mcO_F$-module of signature $(1,n-1)$ and $ρ$ is a \emph{framing}, i.e.\! an $\mcO_F$-linear quasi-isogeny of height $0$
$$ρ:X\times_S \overbar{S}→\mbX_n\times_\mbF\overbar{S}$$
such that $ρ^* λ_{\mbX_n} = λ$. An isomorphism between two tuples $(X,i,λ,ρ)$ and $(X',i',λ',ρ')$ is an isomorphism $γ:X\rightarrow X'$ such that $ρ = ρ'\circ γ$. It is automatically $\mcO_F$-linear and satisfies $γ^*λ'=λ$.

By the results of \cite{RZ}, the functor $\mcN_n$ is representable by a formal scheme, formally of finite type and formally smooth of dimension $n-1$ over $\Spf \mcO_{\breve F}$. Its structure was examined in \cite{vollaardwedhorn}. In particular, $\mcN_1\iso \Spf \mcO_{\breve F}$ and $\mcN_2\iso \Spf \mcO_{\breve F}[[t]]$. The universal object $\mcY$ over $\mcN_1$ is called the \emph{canonical lift}. We let $\overbar{\mcY}$ denote the same group with Galois conjugate $\mcO_F$-action. It lifts $\overbar{\mbY}$. In general, we denote by $\mcX_n$ the universal object over $\mcN_n$.

Let us denote the Rosati involution of a polarized $p$-divisible group by $x\mapsto x^*$. Equation \eqref{signature} can then be written as $i(a)^*=i(\overbar{a})$. It is known that $D:=\End^0_{\mcO_{F_0}}(\mbY)$ is a quaternion division algebra over $F_0$ with maximal order $\mcO_D:=\End_{\mcO_{F_0}}(\mbY)$. By definition of $\mbX_n$, there is an identification
$$\End^0_{\mcO_{F_0}}(\mbX_n) = M_n(D).$$
The action of $\mcO_F$ on $\mbY$ induces inclusions $F\subset D$ and $M_n(F)\subset M_n(D)$.

We fix an element $\varpi\in \mcO_D$ such that $\varpi^*=-\varpi$, $\varpi^2=π$ and $\varpi a = \overbar{a}\varpi$ for all $a\in F$. The subset of $F$-linear homomorphisms $\mbY\rightarrow \overbar{\mbY}$ can then be identified with $\varpi \mcO_F\subset D=\End^0_{\mcO_{F_0}}(\mbY)$. Thus we get an identification
\begin{equation}
\begin{aligned}\label{Flinear}
M_n(D)\supset M_n(F)& \overset{\iso}{\longrightarrow} \End^0_{\mcO_F}(\mbX_n)\\
x&\longmapsto \diag(\varpi,1,\ldots,1) \cdot x\cdot \diag(\varpi^{-1},1,\ldots,1).\end{aligned}
\end{equation}

This isomorphism identifies $U(F_0)$ with $G(F_0)$, the group of $F$-linear quasi-isogenies of $\mbX_n$ which preserve the polarization. Similarly, $\mfu(F_0)$ is identified with $\mfg(F_0):=\Lie G(F_0)$. From now on, we write $\varpi x\varpi^{-1}$ for this conjugation.

The group $\mcX_{n-1}\times \overbar{\mcY}$ over $\mcN_{n-1}$ induces a closed immersion $δ:=\mcN_{n-1}→\mcN_n$. Its graph $Δ$ equals the locus in $\mcN_{n-1}\times_{\Spf \mcO_{\breve F}} \mcN_n$ on which the isomorphism between the framing objects
$$\mbX_{n-1}\times \overbar{\mbY} \overset{=}{\longrightarrow} \mbX_n$$
lifts to the universal objects
$$\mcX_{n-1} \times \overbar{\mcY} \overset{\iso}{\longrightarrow} \mcX_n.$$

Given any $x\in \End^0_{\mcO_F}(\mbX_n)$, we define the ``translated diagonal'' $Δ_x$. It is the locus in $\mcN_{n-1}\times_{\Spf \mcO_{\breve F}} \mcN_n$ on which the quasi-homomorphism
$$\mbX_{n-1}\times \overbar{\mbY} = \mbX_n \overset{x}{\longrightarrow}\mbX_n$$
between the framing objects lifts to the universal objects
$$\mcX_{n-1}\times \overbar{\mcY} → \mcX_n.$$
We emphasize that the above morphisms are not endomorphisms. Namely the source comes from the left factor of $\mcN_{n-1}\times_{\mcO_{\breve F}} \mcN_n$, while the target comes from the right factor. 

The closed formal subscheme $Δ_x\subset \mcN_{n-1}\times \mcN_n$ has an alternative interpretation if $x$ preserves the polarization of $\mbX_n$, i.e.\! if $x\in G(F_0)$. Such an $x$ is a quasi-automorphism of the framing object $(\mbX_n,i_{\mbX_n},λ_{\mbX_n})$ and induces an automorphism of $\mcN_n$,
$$x:(X,i,λ,ρ)\longmapsto (X,i,λ,x\circ ρ).$$
The cycle $Δ_x$ is then the graph of the composition $x\circ δ:\mcN_{n-1}→\mcN_n$.

For $x\in M_n(F)$, we write $Δ_x:=Δ_{\varpi x \varpi^{-1}}$ using the isomorphism \eqref{Flinear}.
\begin{lem}[Zhang, \protect{\cite[Lemma 2.8]{zhang1}}]\label{lemfininter}
Assume that $F_0=\mbQ_p$. For regular semi-simple $x\in M_n(F)$, the schematic intersection $Δ\cap Δ_x$ is a projective scheme over $\Spf \mcO_{\breve F}$. In particular, the Euler-Poincaré characteristic $χ(\mcO_{Δ} \tensor^{\mbL} \mcO_{Δ_x})$ is finite.
\end{lem}

\begin{rmk}
The proof of Lemma \ref{lemfininter} relies on a global argument in \cite{KRglobal} which needs $F_0=\mbQ_p$. But it is expected that the Lemma holds for any base field. This is why we keep our more general notation.  Also note that our computation in Section 3 implicitly proves Lemma \ref{lemfininter} for any $F_0$ in the case $n=3$.
\end{rmk}

In the context of the previous lemma, we define the intersection product of $Δ$ and $Δ_x$ as
\begin{equation}
\label{intersection}\langle Δ,Δ_x\rangle := χ(\mcO_Δ\tensor ^{\mbL} \mcO_{Δ_x}).
\end{equation}
We can now state both versions of the AFL conjecture.

\begin{conj}[AFL, group version]\label{conjaflgrp}
Let $γ\in S(F_0)_{\mathrm{rs}}$ match some $g\in U(F_0)\iso G(F_0)$. Then
$$-Ω(γ)\partial O_γ(1_{S(\mcO_{F_0})}) = \log(q)\langle Δ, Δ_g\rangle.$$
\end{conj}

\begin{conj}[AFL, Lie algebra version]\label{conjafllie}
Let $y\in \mfs(F_0)_{\mathrm{rs}}$ match some $x\in \mfu(F_0)\iso \mfg(F_0)$ and assume that the schematic intersection $Δ\cap Δ_x$ is zero-dimensional. Then
$$-ω(y)\partial O_y(1_{\mfs(\mcO_{F_0})}) = \log(q)\mathrm{len}(\mcO_{Δ\cap Δ_x}).$$
\end{conj}

We now explain the relation between these two conjectures. First note that the two families of cycles $Δ_g$ and $Δ_x$ are quite different from each other. As explained above, the cycle $Δ_g$ is really a translate of $Δ$.
In particular, it is isomorphic to $\mcN_{n-1}$ as a formal scheme and
$$\dim Δ_g = (n-1) = \frac{1}{2} \dim (\mcN_{n-1}\times _{\Spf \mcO_{\breve F}}\mcN_n).$$
Thus it makes sense to define the right hand side of Conjecture \ref{conjaflgrp} with intersection theory \eqref{intersection}.

By contrast, an element $x\in \mfu(F_0)$ is only an $\mcO_F$-linear quasi-endomorphism of $\mbX_n$ which does not necessarily preserve the polarization. The structure of the cycle $Δ_x$ is not at all clear. For example if $m\geq 1$, then $Δ_x\subset Δ_{p^mx}$ and
$$\mcN_{n-1}\times \mcN_n=\lim_{\rightarrow} Δ_{p^mx}$$
as formal schemes. It may happen that $Δ_x$ is not of middle dimension $(n-1)$. For example if $n=2$, then $Δ_x$ is very often a zero-dimensional scheme and the intersection product $\langle Δ,Δ_x \rangle$ will vanish (although the left hand side of the AFL does not). Thus one cannot define the right hand side of Conjecture \ref{conjafllie} with definition \eqref{intersection}.

We will show in Section 2 that the two AFL conjectures are equivalent in the case of non-degenerate intersection\footnote{At least if $n+2\leq q$.}. In this case, the right hand side of Conjecture \ref{conjaflgrp} equals $\log(q)\mathrm{len}(\mcO_{Δ\cap Δ_g})$, see \cite[Proposition 4.2]{RTZ}. We will then compare the schematic intersections $Δ\cap Δ_g$ and $Δ\cap Δ_x$.

In the case $n=3$, the intersection $Δ\cap Δ_x$ is always zero-dimensional. So in this case, the AFL for Lie algebras is equivalent to the original AFL.

\subsection*{Acknowledgments}
I would like to thank M. Rapoport for suggesting to think about the AFL and many helpful discussions.

\section{Reduction to the Lie algebra}
The Cayley transform is an involution on an open subset of $M_n(F)$. It interchanges $U(F_0)$ and $\mfu(F_0)$, as well as $S(F_0)$ and $\mfs(F_0)$. It was introduced in our context by Wei Zhang in \cite[Section 3]{zhang3}. We first recall some of his results. Then we use the Cayley transform to prove the equivalence of the AFL conjectures in the case of non-degenerate intersection, see Theorem \ref{thmreduct}.

\subsection{The Cayley transform}
We consider the $n\times n$-matrices $M_{n,F}:=\Res_{F/F_0}M_n$ as variety over $F_0$. For each $λ\in F^\times$ we let $D_λ$ be the closed subvariety
$$D_λ=\{x\in M_{n,F}\mid \det(x-λ)=0\}.$$
The Cayley transform is the automorphism of $M_{n,F}\setminus D_1$ defined by
$$\mfc:x\mapsto -(1+x)(1-x)^{-1}.$$
For $λ,κ\in F^\times$, we consider the modified Cayley transform
$$\begin{aligned}
{}_λ\mfc_κ :M_{n,F}\setminus D_{κ}&\overset{\iso}{\longrightarrow} M_{n,F}\setminus D_{λ}\\
x&\longmapsto  -λ(κ+x)(κ-x)^{-1}.
\end{aligned}$$
Its inverse is ${}_κ\mfc_λ$. For varying $λ$ and $κ$, these transforms cover $M_{n,F}$ in source and target.

Consider a subgroup $G\subset\Res_{F/F_0}GL_n$, with the case of interest being $\mathrm{Res}_{F/F_0}GL_{n-1}$ (embedded in the upper left). It acts on $M_{n,F}$ by conjugation. The subvarieties $D_λ$ are stable under $G$ and the ${}_λ\mfc_κ$ are equivariant for this action. In particular, $x\in M_n(F)$ is regular semi-simple if and only if ${}_λ\mfc_κ(x)$ is.

\subsection{Cayley transform on the geometric side}
\begin{lem}[Zhang]
Let $λ\in F^1$ and $κ\in \mcO_{F_0}^\times$. The Cayley transform ${}_λ\mfc_κ$ restricts to an isomorphism
$$(\mfu\setminus D_{κ})(F_0) → (U\setminus D_λ)(F_0)$$
which preserves the property ``regular semi-simple''.
\end{lem}
\begin{proof}
Multiplication by $λ$ and $κ$ preserve $U(F_0)$ and $\mfu(F_0)$, respectively. So it is enough to consider $\mfc$ itself. Let $x\in \mfu(F_0)$, i.e. ${}^t\overbar{x}=-JxJ^{-1}$. Then
$$\begin{aligned}
\mfc(x)^*J \mfc(x)&=(1-{}^t\overbar{x})^{-1}(1+{}^t\overbar{x})J(1+x)(1-x)^{-1}\\
&=(1-{}^t\overbar{x})^{-1}(J+{}^t\overbar{x}Jx)(1-x)^{-1}\\
&=(1-{}^t\overbar{x})^{-1}J(1-x^2)(1-x)^{-1}\\
&=(J^{-1}+xJ^{-1})^{-1}(1+x)\\
&=J.\end{aligned}$$

Verifying that ${}_κ\mfc_λ$ maps the group to its Lie algebra is similar. The claim about the property ``regular semi-simple'' was explained before.
\end{proof}
To avoid confusion, we write ${}_κ\mfc^{-1}_λ$ for the Cayley transform from the group to the Lie algebra.

In the following, we consider $x\in \mfu(F_0)_{\mathrm{rs}}$ or $g\in U(F_0)_{\mathrm{rs}}$. In order to describe $Δ\cap Δ_x\subset \mcN_{n-1}\times \mcN_n$, we project to the first factor. This identifies $Δ\cap Δ_x$ with the locus in $\mcN_{n-1}$ on which
$$\varpi x\varpi^{-1} \in \End^0(\mbX_{n-1}\times \overbar{\mbY})$$
lifts to an endomorphism of $\mcX_{n-1}\times \overbar{\mcY}$. By definition, $\mcX_{n-1}\times \overbar{\mcY}$ is an $\mcO_F$-module and $\varpi x\varpi^{-1}$ is $\mcO_F$-linear. It is then clear that $Δ\cap Δ_x$ equals the locus on which the action of the whole algebra $\mcO_F[x]\subset M_n(F)$ lifts. Similarly, $Δ\cap Δ_g$ is the locus in $\mcN_{n-1}$ to which the action of $\mcO_F[g]\subset M_n(F)$ lifts.

We say that $x$ (resp.\! $g$) is integral at a point $(X,i,λ,ρ)\in \mcN_{n-1}(\mbF)$ if $(X,i,λ,ρ)\in Δ\cap Δ_x$ (resp.\! $(X,i,λ,ρ)\in Δ\cap Δ_g$). This is equivalent to $\varpi x\varpi^{-1}$ (resp.\! $\varpi g\varpi^{-1}$) lying in $R_X:=ρ\End_{\mcO_F}(X)ρ^{-1}$. Note that $R_X$ is an order in $\End^0_{\mcO_F}(\mbX)$. So if $x$ is integral at $X$, then its characteristic polynomial has coefficients in $\mcO_F$.

\begin{lem}\label{lemcycles}
Assume that $n+2\leq q$ and let $x\in \mfu(F_0)_{\mathrm{rs}}$ and $g\in U(F_0)_{\mathrm{rs}}$. There exists $κ\in \mcO_{F_0}^\times$ such that $\mfc_κ$ is defined at $x$ and $Δ\cap Δ_x = Δ\cap Δ_{\mfc_κ(x)}$. Similarly, there exists $λ\in F^1$ such that $Δ\cap Δ_g= Δ\cap Δ_{\mfc^{-1}_λ(g)}$.
\end{lem}
\begin{proof}
If $(Δ\cap Δ_x)(\mbF)=\emptyset$, there is nothing to prove. So let us assume that this set is non-empty. By the discussion above, it is enough to show the existence of $κ$ and $λ$ such that $\mcO_F[x]=\mcO_F[\mfc_κ(x)]$ and $\mcO_F[g]=\mcO_F[\mfc^{-1}_λ(g)]$. Let us show the existence of $κ$, the other case being similar.

As explained before, the characteristic polynomial of $x$ has integral coefficients. Let $α_1,\ldots,α_n$ be the eigenvalues of $x$, each satisfying $v(α_i)\geq 0$. Since $n< q-1$, there exists $κ\in \mcO_{F_0}^\times$ such that
$$v(α_i-κ)=0\ \ \forall i.$$
It follows that $\det(κ-x)\in \mcO_F^\times$ and hence that $κ-x \in \mcO_F[x]^\times$ by the Cayley-Hamilton theorem. In particular, $\mfc_κ(x)\in \mcO_F[x]$.

The eigenvalues of $\mfc_κ(x)$ are $-(κ+α_i)/(κ-α_i)$. The denominator in $x = {}_κ\mfc^{-1}(\mfc_κ(x))$ has then eigenvalues $1+(κ+α_i)/(κ-α_i)$ which satisfy
$$v\left(1+\frac{κ+α_i}{κ-α_i}\right)=v(2κ)=0\ \ \forall i.$$
It follows by the same arguments as above that $x\in \mcO_F[\mfc_κ(x)]$ which implies $\mcO_F[x]=\mcO_F[\mfc_κ(x)]$.

As mentioned above, the proof of the second statement is completely analogous.
\end{proof}

\subsection{Cayley transform on the analytic side}
\begin{lem}[Zhang]
Let $λ\in F^1$ and $κ\in \mcO_{F_0}^\times$. The Cayley transform ${}_λ\mfc_κ$ restricts to an isomorphism
$$(\mfs\setminus D_{κ})(F_0) → (S\setminus D_λ)(F_0)$$
which preserves the property ``regular semi-simple''.
\end{lem}
\begin{proof}
Multiplication by $λ$ and $κ$ preserve $S(F_0)$ and $\mfs(F_0)$, respectively. So it is enough to consider $\mfc$ itself. Let $y\in \mfs(F_0)$ i.e. $\overbar{y} = -y$. Then
$$(1+y)(1-y)^{-1}(1+\overbar{y})(1-\overbar{y})^{-1} = 1.$$

Conversely if $γ\overbar{γ}=1$, then
$$(1+\overbar{γ})(1-\overbar{γ})^{-1} = (1+γ^{-1})(1-γ^{-1})^{-1} = -(1+γ)(1-γ)^{-1}.$$
(We inserted $γγ^{-1}$ between the brackets in the last step.)

The claim about the regular semi-simplicity follows from the general remarks above.
\end{proof}
From now on, we denote the Cayley transformation from the symmetric space $S$ to the tangent space $\mfs$ by $\mfc^{-1}$. Recall that the Cayley transform is equivariant for the conjugation by $GL_{n-1}(F)$. This implies that if $γ\in (S\setminus D_κ)(F_0)_{\mathrm{rs}}$ matches $g\in U(F_0)$, then $g\notin D_κ$ and ${}_λ\mfc^{-1}_κ(γ)$ matches ${}_λ\mfc^{-1}_κ(g)$. Analogous statements are true when $γ$ is replaced with $g,y$ or $x$.

Let $f$ be a smooth compactly supported function on $S(F_0)$ and $γ\in (S\setminus D_κ)(F_0)_{\mathrm{rs}}$. It is clear that if $γ\notin D_λ$, then
$$O_γ(f,s)=O_{\!{}_κ\mfc^{-1}_λ(γ)}({}_κ\mfc^{-1}_{λ,*}f,s).$$
Namely the Cayley transform is conjugation equivariant and the twisting character $η|\cdot|^s$ in the integrand of $O_γ$ only depends on the variable $h\in GL_{n-1}(F_0)$.

\begin{prop}\label{correspond}
Assume that $n+2\leq q$. For any $γ\in S(F_0)_{\mathrm{rs}}$ with integral characteristic polynomial, there exists $λ\in F^1$ such that $γ\notin D_λ$ and such that there is an equality
$$Ω(γ)\partial O_γ(1_{S(\mcO_{F_0})})=ω(\mfc^{-1}_λ(γ))\partial O_{\mfc^{-1}_λ(γ)}(1_{\mfs(\mcO_{F_0})}).$$

For any $y\in \mfs(F_0)_{\mathrm{rs}}$ with integral characteristic polynomial, there exists $κ\in \mcO_{F_0}^\times$ such that $y\notin D_κ$ and such that there is an equality
$$ω(y)\partial O_y(1_{\mfs(\mcO_{F_0})})=Ω(\mfc_κ(y))\partial O_{\mfc_κ(y)}(1_{S(\mcO_{F_0})}).$$
\end{prop}
\begin{proof}
\emph{Comparison of transfer factors:}\\
First note that $-(λ+γ)(λ-γ)^{-1} = 1-2λ(λ-γ)^{-1}$. We compute
$$\begin{aligned}
ω(\mfc_λ(γ))\ &= η(\det((1-2λ(λ-γ)^{-1})^i\cdot e)_{i=0,\ldots,n-1})\\
&= η(\det((-2λ(λ-γ))^{-i}\cdot e))\\
&= η(\det((λ-γ)^{-i}\cdot e))\\
&= η(\det(λ-γ))^{1-n}\cdot η(\det((λ-γ)^{n-1-i}\cdot e))\\
&= η(\det(λ-γ))^{1-n}\cdot η(\det(γ^i\cdot e))\\
&= η(\det(λ-γ))^{1-n}Ω(γ).\end{aligned}$$
Let $α_1,\ldots,α_n$ be the eigenvalues of $γ$. Since $n\leq q$, we can choose $λ\in F^1$ such that $v(λ-α_i)=0\ \forall i$. Then $\det(λ-γ)$ is a $p$-adic unit and hence $ω(\mfc_λ(γ))=Ω(γ)$.

\emph{Comparison of test functions:}\\
We check that with the above choice of $λ$,
$$\partial O_γ(1_{S(\mcO_{F_0})})=\partial O_{\mfc^{-1}_λ(γ)}(1_{\mfs(\mcO_{F_0})}).$$
More precisely, we check that
$$h^{-1}γh\in S(\mcO_{F_0})\ \ \Leftrightarrow\ \ h^{-1}\mfc^{-1}_λ(γ)h \in \mfs(\mcO_{F_0}).$$

So assume that $h^{-1}γh$ is integral. Then $λ+h^{-1}γh$ and $λ-h^{-1}γh$ are also integral. The latter one has $\det(λ-h^{-1}γh)=\prod (λ-α_i)\in \mcO_F^\times$, so lies in $GL_n(\mcO_F)$. It follows that $\mfc^{-1}_λ(γ)$ is integral.

Conversely if $h^{-1}\mfc^{-1}_λ(γ)h$ is integral, we argue in the same way. Namely
$$γ=-λ(1+\mfc^{-1}_λ(γ))(1-\mfc^{-1}_λ(γ))^{-1}$$
and $1-\mfc^{-1}_λ(γ)$ has eigenvalues $1+(λ+α_i)(λ-α_i)^{-1}$. But
$$v\left( 1+\frac{λ+α_i}{λ-α_i}\right)=v(2λ)=0\ \ \forall i.$$
So $1-h^{-1}\mfc^{-1}_λ(γ)h$ is not only integral, but lies in $GL_n(\mcO_F)$. It follows that
$$h^{-1}γh = {}_λ\mfc(h^{-1}\mfc^{-1}_λ(γ)h)$$
is integral.

The analogous statement for $y$ is proved with the same arguments. In this case, $κ\in \mcO_{F_0}^\times/(1+π\mcO_{F_0})$ which has only $q-1$ elements. So we need the stronger relation $n\leq q-2$.
\end{proof}

\begin{thm}\label{thmreduct}
Assume that $n+2\leq q$. Then the AFL for groups (Conjecture \ref{conjaflgrp}) implies the AFL for Lie algebras (Conjecture \ref{conjafllie}). Conversely, the AFL for Lie algebras implies the AFL for groups for all group elements $g\in U(F_0)_{\mathrm{rs}}$ such that $Δ\cap Δ_g$ is $0$-dimensional. 
\end{thm}
\begin{proof}
Let us assume the Lie algebra version of the AFL. This means that for all $x\in \mfu(F_0)_{\mathrm{rs}}$ such that $\dim Δ\cap Δ_x = 0$ with match $y\in \mfs(F_0)_{\mathrm{rs}}$, there is an equality
$$-ω(y)\dOrb_y(1_{\mfs(\mcO_{F_0})}) = \log(q)\len_{\mcO_{\breve F}} \mcO_{Δ\cap Δ_x}.$$

Now let $g\in U(F_0)_{\mathrm{rs}}$ be such that $\dim(Δ\cap Δ_g)=0$. Let $γ\in S(F_0)_{\mathrm{rs}}$ be a match for $g$. We have to show the equality
$$-Ω(γ)\partial O_γ(1_{S(\mcO_{F_0})}) = \log(q)\langle Δ,Δ_g\rangle.$$

If $γ$ is not $GL_n(F)$-conjugate to an integral matrix, both sides are $0$. So let us assume that $γ$ is conjugate to an integral matrix. In particular, its characteristic polynomial has integral coefficients. We choose $λ\in F^1$ as in Proposition \ref{correspond}. Then
$$Ω(γ)\partial O_γ(1_{S(\mcO_{F_0})}) = ω(\mfc^{-1}_λ(γ))\partial O_{\mfc^{-1}_λ(γ)}(1_{\mfs(\mcO_{F_0})}).$$

Now $g\notin D_λ$ and the element $\mfc^{-1}_λ(g)$ matches $\mfc^{-1}_λ(γ)$. By the proof of Lemma \ref{lemcycles}, the same $λ$ also satisfies
$$Δ\cap Δ_g=Δ\cap Δ_{\mfc^{-1}_λ(g)}.$$
By \cite[Proposition 4.2]{RTZ}, all higher $Tor$-terms in the intersection product $\langle Δ,Δ_g\rangle$ vanish. So
$$\langle Δ,Δ_g\rangle = \len Δ\cap Δ_g=\len Δ\cap Δ_{\mfc^{-1}_λ(g)}.$$
Hence the two sides of the AFL for groups equal the two sides of the AFL for Lie algebras. This proves the statement.

Conversely, we can apply the same arguments to deduce the Lie algebra version from the group version.
\end{proof}

\section{The Lie algebra version for $n=3$}
We keep all previous notation, but specialize to $n=3$. We fix a regular semi-simple element $x\in \mfu(F_0)_{\mathrm{rs}}$ matching some $y\in \mfs(F_0)_{\mathrm{rs}}$ and want to verify the AFL (Conjecture \ref{conjafllie}) for it. As mentioned in the introduction,
$$Δ\iso\mcN_2\iso \Spf \mcO_{\breve F}[[t]]$$
is geometrically a point. So the intersection $Δ\cap Δ_x$ (resp.\! $Δ\cap Δ_g$) is automatically zero-dimensional. It follows that a verification of the AFL for Lie algebras will also prove the AFL for groups in this case, see Theorem \ref{thmreduct}.

\subsection{Choosing coordinates}
For a vector $j={}^t(j_1,j_2)\in F^2$, we write $v(j)$ for the valuation $v(-πj_1\overbar{j_1} + j_2\overbar{j_2})$ of its norm with respect to the hermitian form $J$.

By \eqref{liealg}, the element $x$ has the form
$$x=\left(\begin{matrix}
a_1 & b & j_1 \\
π\overbar{b} & a_2 & j_2\\
π\overbar{j_1} & -\overbar{j_2} & d
\end{matrix}\right)$$
with the property that $a_1,a_2,d\in τF_0$. We also set $j:=j(x):={}^t(j_1,j_2)$. Since $x$ is regular semi-simple, $j(x)\neq 0$.

Let $\tilde U(F_0)$ denote the unitary group associated to the form defined by $\diag(-π,1)$. It acts on $\mfu(F_0)$ by conjugation without changing the matching relation or the cycle $Δ_x$. In particular, we can assume that ${}^tj=(0,π^m)$ if $v(j)=2m$ is even or ${}^tj=(π^m,0)$ if $v(j)=2m+1$ is odd. Note that this forces $b\neq 0$, since otherwise $x$ would have a non-trivial stabilizer in $GL_2(F)$.

We choose to work with the match $y\in \mfs(F_0)_{\mathrm{rs}}$ for $x$ given by
\begin{equation}\label{matchA}
y= \begin{cases} \left(\begin{matrix}
a_1 & τπN(b) & τπ^m \\
τ^{-1} & a_2 & \\
τ^{-1}ππ^m &  & d
\end{matrix}\right) & \tif v(j) = 2m+1 \text{ is odd}\\

\left(\begin{matrix}
a_1 & -τ^{-1} &  \\
-τπN(b) & a_2 & -τπ^m\\
 & τ^{-1}π^m & d
\end{matrix}\right) & \tif v(j) = 2m \text{ is even.}\end{cases}
\end{equation}
Here we conjugated by $\diag(τ,π^{-1}\overbar{b}^{-1},1)$ in the odd case and by $\diag(b^{-1},-τ,1)$ in the even case. This specific choice is motivated by the fact that both cases can be treated with the same equations later.

The endomorphism $\varpi x\varpi^{-1}$ of $\mbX_3$ is given by
$$\varpi x\varpi^{-1}=\begin{cases}\left(\begin{matrix}
-a_1 & \varpi b & \varpi π^m \\
\varpi b & a_2 & \\
\varpi π^m & & d
\end{matrix}\right) & \tif v(j)=2m+1\\

\left(\begin{matrix}
-a_1 & \varpi b &  \\
\varpi b & a_2 & π^m\\
& -π^m & d
\end{matrix}\right) & \tif v(j)=2m. \end{cases}$$

Recall that $\mcX_2$ denotes the universal object over $\mcN_2\iso Δ$. Also recall that $Δ\cap Δ_x$ can be identified with the locus in $\mcN_2$ to which $\varpi x\varpi^{-1}$ lifts as endomorphism of $\mcX_2\times \overbar{\mcY}$.

The unique geometric point of $\mcN_2$ corresponds to the group $\mbX_2=\mbY\times \overbar{\mbY}$ over $\mbF$. So $Δ\cap Δ_x\neq \emptyset$ if and only if $\varpi x\varpi^{-1}\in M_3(\End(\mbY))=M_3(\mcO_D)$. This is equivalent to $x\in M_3(\mcO_F)$.

\begin{lem}
The element $x$ is integral if and only if it matches an integral element in $\mfs(F_0)_{\mathrm{rs}}$. In particular, the AFL for $n=3$ holds if $x$ is not integral.
\end{lem}
\begin{proof}
If $x$ is integral, then our $y$ from \eqref{matchA} is also integral. Conversely, let us assume that there is $h\in GL_2(F_0)$ such that $hyh^{-1}$ is integral. (Every other match is of this form.) We show that $x$ is also integral.

First note that the bottom right $d$ is unchanged under conjugation by $GL_2(F)$. So we can assume $d=0$. From now on, we restrict to the case of odd $v(j)$.

We can multiply $h$ on the left by $GL_2(\mcO_{F_0})$ without changing the integrality. So let us assume that $h$ is upper triangular,
$$h=\left(\begin{matrix} µ & \star\\ & ν\end{matrix}\right),$$
$$hyh^{-1}= \left(\begin{matrix}
y_1 & y_2 & µτπ^m \\
y_3 & y_4 & \\
µ^{-1}τ^{-1}ππ^m &  & 
\end{matrix}\right).$$
Then both $µτπ^m$ and $µ^{-1}τ^{-1}ππ^m$ are integral, so $m\geq 0$. Now $\det(x)=-π^{2m+1}a_2 = \det(hyh^{-1}) = -π^{2m+1}y_4$ which implies $a_2=y_4$. Arguing with the trace yields $y_1=a_1$. Since also the determinant of the $2\!\times\! 2$-block is preserved, $πN(b)=y_2y_3$. So all of $a_1,a_2,b$ are integral.

The case when $v(j)$ is even is done similarly with lower triangular $h$.
\end{proof}

From now on we restrict to the case of integral $x$. The diagonal action of $τ\mcO_{F_0}\times τ\mcO_{F_0}$ on $\mbX_2\times\overbar{\mbY}$ lifts to $\mcX_2\times \overbar{\mcY}$, so we can subtract it from $x$ without changing $Δ\cap Δ_x$. This also does not change the derived orbital integral $ω(y)\partial O_y(1_{\mfs(\mcO_{F_0})})$. More precisely, let $µ\in τ\mcO_{F_0}\times τ\mcO_{F_0}$. Then $x$ is regular semi-simple if and only if $x-µ$ is, and $x-µ$ matches $y-µ$ if and only if $x$ matches $y$. Furthermore $O_y(1_{\mfs(\mcO_{F_0})},s)=O_{y-µ}(1_{\mfs(\mcO_{F_0})},s)$ for all $s$ and $ω(y)=ω(y-µ)$. In particular, it suffices to verify the AFL for $x-µ$.

So from now on we assume that $d=0$ and that the matrix trace $\mathrm{tr}(x)=0$ vanishes.

\subsection{Geometric Side}
We now want to compute $\len(Δ\cap Δ_x)$. In the previous subsection we reduced to a quasi-endomorphism $\varpi x\varpi^{-1}$ of the form
$$\varpi x \varpi^{-1}=\begin{cases}\left(\begin{matrix}
a & \varpi b & \varpi π^m \\
\varpi b & a & \\
\varpi π^m  & & 
\end{matrix}\right) & \mathrm{if}\ v(j)=2m+1\ \text{is odd}\\

\left(\begin{matrix}
a & \varpi b &  \\
\varpi b & a & π^m\\
& -π^m & 
\end{matrix}\right) & \mathrm{if}\ v(j)=2m\ \text{is even}.\end{cases}$$

This is an element of the matrix ring
$$\left(\begin{matrix}
\End^0(\mbX_2) & \Hom^0(\overbar{\mbY}, \mbX_2)\\
\Hom^0(\mbX_2,\overbar{\mbY}) & \End^0(\overbar{\mbY})
\end{matrix}\right).
$$
We can compute the locus where it lifts to $\mcX_2\times \overbar{\mcY}$ entry-wise. The lower right entry of $\varpi x\varpi^{-1}$ is $0$, so lifts to all of $\mcN_2$. The upper right $\diag(\varpi,1) j$ and lower left ${}^t\overbar{j}\!\diag(\varpi^{-1},1)$ are dual to each other under the Rosati involution (up to sign). So they lift to the same locus.

The crucial point in the computation happens now. The vector $j$ is a homomorphism $\overbar{\mbY}→\mbX_2$, i.e.\! a special homomorphism in the sense of Kudla and Rapoport, see \cite{kudlarapoport}. We denote by $\mcZ(j)\subset \mcN_2$ the closed formal subscheme to which it lifts, a so-called special cycle. By \cite[Theorem 8.1]{kudlarapoport}, it is a divisor and there is an equality,
$$\mcZ(j)=\sum^{v(j)}_{s=0,\ s\equiv v(j)\mod 2} \mcZ_s$$
where $\mcZ_s\subset \mcN_2$ is the quasi-canonical divisor of level $s$. It is isomorphic to $\Spf W_s$, where $W_s$ is the ring of integers of the ring class field $F_s/\breve{F}$ associated to the order $\mcO_s:=\mcO_{F_0}+π^s\mcO_F$.

The inclusion $\mcZ(j)→\mcN_2$ is induced from the Serre construction applied to a quasi-canonical lift $\mcY_s/W_s$. More precisely, $\mcO_F\tensor_{\mcO_{F_0}} \mcY_s$ is a unitary $p$-divisible $\mcO_F$-module with framing
\begin{equation}\label{framing}
ψ:\mcO_F\tensor_{\mcO_{F_0}} \mbY\iso \mbY\times\overbar{\mbY},\ µ\tensor x \mapsto (µx,µx).
\end{equation}
We refer to \cite{wewers} for more about quasi-canonical lifts.

Let us denote the upper left entry of $\varpi x\varpi^{-1}$ by
$$z=\left(\begin{matrix} a & \varpi b \\ \varpi b & a\end{matrix}\right).$$
Let us also write $\mcZ_s(z)$ for the locus on the quasi-canonical divisor $\mcZ_s$ where $z$ is an endomorphism. The decomposition above yields
\begin{equation}\label{decomp}
\len(Δ\cap Δ_x) = \sum_{s=0,\ s\equiv v(j)}^{v(j)} \len \mcZ_s(z).
\end{equation}
We now explain how to compute each of the lengths $\len \mcZ_s(z)$.

First note that $z$ is a quasi-endomorphism of $\mcY_s$ in the coordinates given by $ψ$. These coordinates do not lift to $\mcO_F\tensor\mcY_s$. Instead we use the coordinates given by choosing the $\mcO_{F_0}$-basis $1,τ$ of $\mcO_F$,
$$\mcO_F\tensor\mcY_s = \mcY_s \times τ\mcY_s$$
These coordinates induce a similar decomposition $φ:\mcO_F\tensor \mbY = \mbY\times τ\mbY$ on the special fiber. (The bars don't play a role here.) The framing \eqref{framing} is then given by the matrix
\begin{equation}
\mbY\times τ\mbY\ \overset{\left(\begin{smallmatrix} 1 & τ\\ 1 & -τ\end{smallmatrix}\right)}{\longrightarrow}\ \mbY\times\overbar{\mbY}.
\end{equation}

We rewrite $z$ as $z'$ in these new coordinates (note that $\varpi τ = -τ\varpi$)
$$
\begin{aligned}z' =&  \left(\begin{matrix}
1 & τ\\
1 & -τ\end{matrix}\right)^{-1} z \left(\begin{matrix}
1 & τ\\
1 & -τ\end{matrix}\right)\\
= & \left(\begin{matrix}
a+\varpi b & \\
 & a+\varpi b
\end{matrix}\right).\end{aligned}$$

To express the relevant lengths, we define
$$a(k)=1+\frac{(q^k-1)(q+1)}{q-1} = 1+q+\ldots+q^k+q^{k-1}+\ldots+q+1.$$

\begin{thm}[Gross-Keating, see \cite{vollaard}]\label{thmgk}
Let $\mcY_s/\Spf W_s$ be a quasi-canonical lift of level $s$. Let $l\geq 0$ and
$$f\in (\mcO_s+\varpi^l \End(\mbY))\setminus (\mcO_s+\varpi^{l+1} \End(\mbY)).$$
Then the length $n$ of the locus on $\Spf W_s$ to which $f$ lifts is equal to
$$n= \begin{cases} a(\frac{l}{2}) & \tif l\leq 2s \text{ is even}\\
                   a(\frac{l-1}{2})+q^{(l-1)/2} & \tif l\leq 2s \text{ is odd}\\
                   a(s-1)+q^{s-1}+\left(\frac{l+1}{2}-s\right)e_s & \tif l\geq 2s-1.\end{cases}$$
Here $e_s$ is the ramification index of $W_s/W$. If $s\geq 1$, then $e_s=q^s+q^{s-1}$.
\end{thm}

To compute $\len\mcZ_s(z)$, we apply this theorem with $f=a+\varpi b$. We define $l:=v_{\varpi}(\varpi b)=2v(b)+1$ and $k:=v_{\varpi}(a)=2v(a)$. Then we get
\begin{equation}\label{eqGK}
\len_{\mcO_{\breve F}} \mcZ_s(z)= \begin{cases}
a\big ( \frac{l-1}{2}\big) + q^{(l-1)/2} & \tif (l<k\ \mathrm{or}\ 2s\leq k)\ \mathrm{and}\ l< 2s\\
a(s-1)+q^{s-1}+\big(\frac{l+1}{2} - s\big)e_s & \tif (l<k\ \mathrm{or}\ 2s\leq k)\ \mathrm{and}\ l>2s\\
a(k/2) & \tif k<l \tand k < 2s.\end{cases}
\end{equation}

\subsection{Analytic Side}
Recall that $y$ denotes the matching candidate for $x$ from \eqref{matchA},
$$
y= \begin{cases} \left(\begin{matrix}
-a & τπN(b) & τπ^m \\
τ^{-1} & a & \\
τ^{-1}ππ^m &  & 
\end{matrix}\right) & \tif v(j) = 2m+1 \text{ is odd}\\

\left(\begin{matrix}
-a & -τ^{-1} &  \\
-τπN(b) & a & -τπ^m\\
 & τ^{-1}π^m & 
\end{matrix}\right) & \tif v(j) = 2m \text{ is even.}\end{cases}
$$
Note that $ω(y)=+1$ in both cases. We denote by $\tilde z$ the upper left $2\!\times \! 2$-block.

We now want to compute
$$-ω(y)\dOrb_y(1_{\mfs(\mcO_{F_0})})=\log(q)\int_{GL_2(F_0)} (-1)^{v(\det h)+1}v(\det h)1_{\mfs(\mcO_{F_0})}(hyh^{-1})dh.$$
(The minus sign comes from the coordinate substitution $h\mapsto h^{-1}$.) Note that $\mfs(\mcO_{F_0})$ is stable under the action of $\GL_2(\mcO_{F_0})$. So this derived orbital integral counts cosets $[h]\in GL_2(\mcO_{F_0})\backslash GL_2(F_0)$ with the property that $hyh^{-1}\in M_3(\mcO_F)$ with certain weights $-η(\det h)v(\det h)\log(q)$. We compute the integral by counting these cosets.

Recall the definitions $l=v(πN(b))$ and $k=v(a^2)$. Let us represent any class $[h]$ in triangular form
$$h=\begin{cases}π^{m+1}π^{-s}
\left(\begin{matrix}
1 & \star\\
& π^t\end{matrix}\right) & \tif v(j) \text{ is odd}\\
π^mπ^{-s}\left(\begin{matrix}
π^t & \\ \star& 1
\end{matrix}\right). & \tif v(j) \text{ is even}.\end{cases}$$
Here $s$ and $t$ are uniquely determined while $\star$ is unique in $F_0/π^t\mcO_{F_0}$. We now determine necessary and sufficient conditions on $s,t$ and $\star$ for $hyh^{-1}$ to be integral.

\emph{$hj$ integral:}
This is equivalent to $s\leq v(j)$.

\emph{${}^tjh^{-1}$ integral:}
This is equivalent to $0 \leq s$ and $\star\in π^{t-s}\mcO_{F_0}$.

\emph{$h\tilde zh^{-1}$ integral:}
We compute in the odd case:
$$\begin{aligned}
h\tilde zh^{-1} = & \left(\begin{matrix}
1 & \star\\ & π^t \end{matrix}\right)
\left(\begin{matrix}
-a & τπN(b)\\τ^{-1} & a
\end{matrix}\right)
\left(\begin{matrix}
1 & -\star π^{-t} \\ & π^{-t}
\end{matrix}\right)\\
= & \left(\begin{matrix}
-a+\starτ^{-1} & π^{-t}(-\star^2τ^{-1} + 2a\star+τπN(b))\\
τ^{-1}π^t & a-\star τ^{-1}
\end{matrix}\right).
\end{aligned}$$

Since $a\in \mcO_F$, the upper left and lower right entries are integral if and only if
$$\star\in \mcO_{F_0}.$$
The lower left is integral if and only if
$$0\leq t.$$
The upper right entry is integral if and only if
$$(\star-τa)^2\in τ^2a^2+τ^2πN(b)+π^{t}\mcO_{F_0}.$$
The reader may check that with our choice of coordinates, the even case leads to exactly the same conditions.

For fixed $s,t$, we denote by $α(s,t)$ the number of classes $\star$ satisfying all these conditions. Taken together, these conditions are
$$0\leq s\leq v(j),\ 0\leq t,\ \star\in π^{\max\{0,t-s\}}\mcO_{F_0}/π^t\mcO_{F_0}\ \tand$$
$$(\star-τa)^2\in τ^2a^2+τ^2πN(b)+π^{t}\mcO_{F_0}.$$
If there are no solutions for $\star$ or if $s,t$ are out of their ranges, we set $α(s,t)=0$. Then the derived orbital integral is
\begin{equation}\label{derivedInt}
-ω(y)\partial O_y(1_{\mfs(\mcO_{F_0})}) = \log(q) \sum_{s,t} (-1)^{t+1}(2m-2s+t+ε)α(s,t)
\end{equation}
with $ε=2$ in the odd and $ε=0$ in the even case. We also define the partial sum
\begin{equation}\label{sigma}
σ(s):= \sum_t (-1)^{t+1}(2m-2s+t+ε)α(s,t).
\end{equation}

\textbf{Case A: $l<k$}\\
(We make the basic assumptions that $0\leq s\leq v(j)$ and $0\leq t$.)\\
Then $v(τ^2a^2+τ^2πN(b))=l$ is odd. In particular it is not a square. We can only get a nontrivial solution count if $πN(b)\in π^{t}\mcO_{F_0}$, i.e.\! if $t\leq l$. In this case,
we simply ask for
$$\star^2\in π^{t}\mcO_{F_0}.$$
Given the restrictions on $\star$, this yields
$$α(s,t)=q^{\min\{t-\lceil t/2\rceil,t-(t-s)\}} = q^{\min\{\lfloor t/2\rfloor,s\}}.$$

Now we compute $σ(s)$ for $0\leq s\leq v(j)$. Note that $l$ is odd, so the alternating sum below has an even number of summands.
\begin{equation}\label{aplusb}
\begin{aligned}σ(s):=&\sum_{t=0}^l (-1)^{t+1}(2m-2s+t+ε)q^{\min\{\lfloor t/2\rfloor,s\}}\\
=&\begin{cases} \frac{l-1}{2} & \tif s=0\\
1+q+\ldots+q^{\lfloor l/2\rfloor} & \tif l < 2s\\
1+q+\ldots+q^s+\frac{1}{2}(l-2s-1)q^s & \totherwise. \end{cases}\end{aligned}
\end{equation}

Then
\begin{equation}\label{casea1}
σ(s)+σ(s-1)=\begin{cases}
2(1+\ldots+q^{\lfloor l/2 \rfloor}) & \tif l < 2s\\
2(1+\ldots+q^{s-1})+(\frac{l-1}{2} - s)e_s & \totherwise.\end{cases}
\end{equation}
(For $s=0$, the formula yields $σ(0)+σ(-1)=σ(0)=(l-1)/2$.)

\begin{lem}
There is an equality $\len \mcZ_s(z)=σ(s)+σ(s-1)$. In particular,
$$\len(Δ\cap Δ_x) = \sum_{s=0,\ s\equiv v(j)}^{v(j)} \len \mcZ_s(z)\ =\ \sum_{s=0}^{v(j)} σ(s)$$
and so the AFL is proven if $l<k$.
\end{lem}
\begin{proof}
This is a combination of \eqref{eqGK}, the decomposition \eqref{decomp} and \eqref{derivedInt}.
\end{proof}

\textbf{Case B: $k<l$}\\
(Again we make the basic assumptions $0\leq s\leq v(j)$ and $0\leq t$.)\\
If $t\leq k$, we apply the same arguments as in Case 1 and get
$$α(s,t) = q^{\min\{\lfloor t/2\rfloor,s\}}.$$
If $k<t\leq l$, we have to solve
$$(1- \star/τa)^2\in 1+π^{t-k}\mcO_{F_0}.$$
There are two classes of square $1$ in $\mcO_{F_0}^\times/(1+π^{t-k}\mcO_{F_0})$. The class of $1$ always contributes
$$q^{\min\{k/2,s\}}.$$
The class of $-1$ only contributes if $t-s\leq k/2$. Namely if $\star$ solves
$$(-1+\star/τa)^2 \in 1+π^{t-k}\mcO_{F_0},$$
then $v(\star)=k/2$. In this case, there are $q^{k/2}$ solutions for $\star$ and we arrive at
\begin{equation}\label{twoA}
α(s,t)=\begin{cases} q^{\min\{k/2,s\}}+q^{k/2} & \tif t\leq s+k/2\\
q^{\min\{k/2,s\}} & \totherwise.\end{cases}
\end{equation}

Finally assume $k<l<t$ and let $β:=πN(b)/(τa)^2$. We have to solve
$$(1- \star/τa)^2\in 1+β+π^{t-k}\mcO_{F_0}.$$
The class of $1+\beta$ in
$$\mcO_{F_0}^\times/(1+π^{t-k}\mcO_{F_0})$$
has two square roots. A solution for $\star$ from the class near $1$ has to have valuation $v(β)+v(τa) = l-k/2$. In particular there are no solutions if $l-k/2 < t-s$. If instead $t\leq s+l-k/2$, then we get $q^{k/2}$ solutions for $\star$ modulo $π^t\mcO_{F_0}$.

With similar arguments, the class close to $-1$ contributes $q^{k/2}$ if and only if $t-s\leq k/2$. We get
\begin{equation}\label{twoB}
α(s,t)=\begin{cases} 2q^{k/2} & \tif t\leq s+k/2.\\
q^{k/2} & \totherwise.\end{cases}
\end{equation}

Again we compute $σ(s)+σ(s-1)$ and compare this number to the contribution of $\mcZ_s$ in \eqref{decomp}. We have
\begin{equation}\label{eqsigma}
\begin{aligned}
σ(s)&=\sum_{t=0}^k(-1)^{t+1}(2m-2s+t+ε)q^{\min\{\lfloor t/2 \rfloor, s\}}\ \ \ \ &(=:A(s))\\
 & + \sum_{t=k+1}^l (-1)^{t+1}(2m-2s+t+ε)q^{\min\{k/2, s\}}&(=:B(s))\\
 & + \sum_{t=k+1}^{s+k/2} (-1)^{t+1}(2m-2s+t+ε) q^{k/2}&(=:C(s))\\
 & + \sum_{t=l+1}^{l+s-k/2} (-1)^{t+1}(2m-2s+t+ε)q^{k/2}&(=:D(s)).
\end{aligned}
\end{equation}
Let us denote the four sums by $A(s),B(s),C(s)$ and $D(s)$ (from top to bottom). The sum $C$ corresponds to the first case in \eqref{twoA} and \eqref{twoB}.

\begin{lem}
There is an equality $\len \mcZ_s(z)=σ(s)+σ(s-1)$. In particular,
$$\len(Δ\cap Δ_x) = \sum_{s=0,\ s\equiv v(j)}^{v(j)} \len \mcZ_s(z)\ =\ \sum_{s=0}^{v(j)} σ(s)$$
and so the AFL is proven if $k<l$.
\end{lem}
\begin{proof} The formula for $\len \mcZ_s(z)$ was given in \eqref{eqGK}. Here we compute $σ(s)+σ(s-1)$.

\emph{Case when $s\leq k/2$:} Here the sums $C$ and $D$ are empty. In the summand of $B$ we always have $\min\{k/2,s\} = s$. Then $(A+B)(s)$ was computed in \eqref{aplusb} and $σ(s)+σ(s-1)$ is given by \eqref{casea1}. The length $\len \mcZ_s(z)$ is given by the second case of \eqref{eqGK}.

\emph{Case when $k/2 < s$:}
First we compute that
$$A(s)+A(s-1) = 2(1+\ldots+q^{k/2-1})+q^{k/2} - (4m-4s+3+2k+2ε)q^{k/2}.$$
Note that in the sum $B$, $\min\{k/2,s\}=k/2$. We let $E(s):=B(s)+D(s)$ so that
$$E(s)= \sum_{t=k+1}^{l+s-k/2}(-1)^{t+1}(2m-2s+t+ε)q^{k/2}.$$
It is now a direct calculation to verify that
$$C(s)+C(s-1)+E(s)+E(s-1) = (4m-4s+3+2k+2ε)q^{k/2}$$
which implies (third case of \eqref{eqGK}) that
$$σ(s)+σ(s-1)=2(1+\ldots+q^{k/2-1})+q^{k/2} = \len \mcZ_s(z).$$

We give some formulas:
$$E(s)=\begin{cases}-\frac{1}{2}(l+s-3k/2)q^{k/2} & \tif s\not\equiv k/2\ \ (2)\\
(2m-2s+k+1+ε)q^{k/2} + \frac{1}{2}(l+s-3k/2-1)q^{k/2} & \tif s\equiv k/2\ \ (2).\end{cases}$$
Similarly
$$C(s)=\begin{cases}(2m-2s+k+1+ε)q^{k/2}+\frac{1}{2}(s-k/2-1)q^{k/2} & \tif s\not\equiv k/2\ \ (2)\\
-\frac{1}{2}(s-k/2)q^{k/2} & \tif s\equiv k/2\ \ (2).\end{cases}$$

This concludes the proof of the AFL for $n=3$.
\end{proof}

\vspace{1cm}
\address{Andreas Mihatsch\\Mathematisches Institut\\der Universität Bonn\\Endenicher Allee 60\\53115 Bonn\\mihatsch@math.uni-bonn.de}
\end{document}